\documentclass[11pt,leqno]{article}

\usepackage{amssymb,amsmath,amsthm,mysects,url}
 \usepackage{graphicx,epsfig}
\newcommand{\bb}[1]{\mathbb{#1}}
\newcommand{\cl}[1]{\mathcal{#1}}

\newcommand{\ovl}{\overline}

\theoremstyle{remark}
\newtheorem{rem}{Remark}
\newtheorem*{rk}{Remark}

\theoremstyle{plain}
\newtheorem{pro}[rem]{Proposition}
\newtheorem{thm}[rem]{Theorem}
\newtheorem{lem}[rem]{Lemma}

\newtheorem{cor}[rem]{Corollary}

\theoremstyle{definition}
\newtheorem*{defn}{Definition}

\begin{document}

\title{On the Dixmier problem\\
(Seminar report after Monod-Ozawa, JFA 2010)}

\author{by\\
Gilles  Pisier\\
Texas A\&M University\\
College Station, TX 77843, U. S. A.\\
and\\
Universit\'e Paris VI\\
Equipe d'Analyse, Case 186, 75252\\
Paris Cedex 05, France}

\maketitle

\begin{abstract}
This seminar report contains a detailed account of the proof of the main results in Monod and Ozawa's recent JFA paper on the Dixmier unitarizability problem. The proof is exactly identical to their proof, but our more pedestrian presentation is hopefully more accessible to nonexperts. This text is not intended for publication
(but it might end up as part of an updated version of \cite{P}).
\end{abstract}

\section{Introduction}\label{sec0}

The   Dixmier problem story starts with the following
result proved in the 
particular case $G={\bb Z}$ by Sz.-Nagy (1947).

\begin{thm}[Day, Dixmier 1950]\label{thm0.1}
Let $G$ be a locally
compact group. If 
$G$ is amenable, then 
$G$ is unitarizable, meaning that every uniformly bounded (u.b.\ in short)
representation 
$\pi\colon \ G\to B(H)$ ($H$ Hilbert) is unitarizable. \end{thm}

 We say that  
$\pi\colon \ G\to B(H)$ ($H$ Hilbert)
 is unitarizable if there exists ${S}\colon\ H\to
H$ invertible such that $t\to {S}^{-1} \pi(t){S}$
is a  unitary representation.  We will mostly
restrict to discrete groups, but otherwise all 
representations $\pi\colon \ G\to B(H)$ are
implicitly assumed continuous on $G$  with
respect to the strong operator topology on
$B(H)$.

In his 1950 paper, Dixmier \cite{Di} asked whether the converse holds:\\
\centerline{\it Are unitarizable groups amenable?}
 
\begin{rem}\label{rem}
It is immediate that unitarizable passes to quotient groups. In the discrete case, it is easy to show
(using induced representations) that it also  passes to subgroups. However, in sharp contrast with amenability, it is unclear whether  the product of two unitarizable groups is unitarizable (but it is
so if one of the two groups is amenable).
\end{rem}
\begin{rk} A more precise version of the preceding Theorem is as follows:
Assume $G$ amenable.
Let
$|\pi| = \sup \{\|\pi(t)\|_{B(H)}\mid t\in G\},$
then, if $|\pi| < \infty$, there exists ${S}\colon\ H\to H$
invertible with 
$\|{S}\| \ \|{S}^{-1}\| \le |\pi|^2$ such that $t\to {S}^{-1}
\pi(t){S}$ is a 
unitary representation. Moreover, $S$ can be found in the von Neumann
algebra generated by the range of $\pi$.\\
 If we stregthen ``unitarizable" by incorporating these extra properties of $S$, then
 we do obtain a characterization of amenability. Namely, we proved (see \cite{P2})
 that if $G$ is amenable and $S$ can always be found such that 
 $\|{S}\| \ \|{S}^{-1}\| \le |\pi|^2$ (or even $\|{S}\| \ \|{S}^{-1}\| \le K|\pi|^\alpha$ for some constant $K$ and some $\alpha<3$) then $G$ is amenable. This holds for
 general locally compact groups. In a different direction, we proved in \cite{P3} that if 
  a {\it discrete} group $G$ is unitarizable and $S$ can always be found in the von Neumann
algebra generated by the range of $\pi$. then $G$ is amenable.
 \end{rk}
 
Recently, Monod and Ozawa proved that if a discrete $G$ is unitarizable and if moreover  the wreath product
$A\wr G$ is unitarizable for some {\it infinite} Abelian group $A$, then $G$ is amenable.
This gives one more characterization of amenability by a strengthened form of unitarizability
(since conversely:  $G$ amenable implies  $A\wr G$ amenable
 and hence unitarizable for any such $A$, and actually for any amenable $A$).

The wreath product $A\wr G$  of two groups is defined as the semi-direct product of $A^{(G)}$ and $G$ when $G$ acts on $A^{(G)}$ by 
   translation say on the left side: More precisely,
the elements of $A^{(G)}$ are functions $a\colon\ G \to A$ that are equal to the unit except in finitely many places, and the action of an element $g\in G$ is defined by $\forall t\in G,  (g\cdot a)(t)=a(g^{-1}t).$\\
The elements of  $A\wr G$ can be viewed as pairs
$(a,g)$ with $a\in A^{(G)}, g\in G$, equipped with a product defined by
$$(a,g) (a',g')=(a  (g\cdot a') , gg').$$
Then $A^{(G)}$ and $G$ can be identified with subgroups of
$A\wr G$,  and $A^{(G)}$ appears as  a normal subgroup
such that for all $g\in G$ and $a\in A^{(G)}$ we have
$g a g^{-1}= g\cdot a$.

In 1955 Ehrenpreis and Mautner gave the first example
(namely $G=SL_2({\bb R})$)  of a non-unitarizable group.
From this follows that any group contaning a copy of ${\bb F}_2$
(free group with 2 generators) is not unitarizable (see \cite{P} for detailed references).

Until recently, all examples of non-unitarizable groups were groups contaning a copy of ${\bb F}_2$ (see \S \ref{Dixsec2} below for  background). This prompted the question raised in \cite{P2} whether there are  non-unitarizable groups among groups without free subgroups, in particular among  the Burnside groups. Monod and Ozawa answered this positively in \cite{MO}, as a consequence of their   result stated above
on  $A\wr G$, see \S \ref{Dixsec4}. Shortly before that, Monod and Epstein (using groups constructed specially for them
by Denis Osin) exhibited in \cite{EM} the first examples of non-unitarizable groups without free subgroups.

\section{Some algebra}\label{Dixsec1}

\noindent 

Let $G$ be a discrete group. Let $\sigma\colon \ G\to B(H)$ be a unitary representation. We define the action of $G$ on $B(H)$ by
\begin{equation}
 g\cdot T = \sigma(g)T \sigma(g)^{-1}.\tag*{$\forall g\in G~~\forall T\in B(H)$}
\end{equation}
Let $H^1_b(G,B(H)) =  C_1/B_1$ where
\begin{align*}
 C_1 &= \{f\in \ell_\infty(G;B(H))\mid f(st) = s\cdot f(t) + f(s)\quad \forall s,t\in G\},\\
B_1 &= \{f\in C_1\mid \exists T\in B(H) \ni f(g)= g\cdot T-T\quad \forall g\in S\}.
\end{align*}
An element $f$ in $C_1$ is called a bounded [1]-cocycle; when $f\in B_1$ the [1]-cocycle is called trivial (it is also called a ``coboundary'').
This terminology is motivated by the following generalization valid for any integer $n\ge 0$:

We denote by ${\cl C}_n$ the space of all bounded functions
$f\colon\ G^{n+1}\to B(H)$ such that $\forall g,g_j\in G$
\[
 f(gg_0,\ldots, gg_n) = g\cdot f(g_0,\ldots, g_n)
\]
and such that $\partial f=0$ where $\partial\colon \ {\cl C}_n\to {\cl C}_{n+1}$ is the linear map defined by:
\begin{equation}
 \partial f(g_0,\ldots, g_{n+1}) = \sum\nolimits^{n+1}_0(-1)^j f(g_0,\ldots, \hat g_j,\ldots, g_{n+1}).\tag*{$\forall g,g_j\in G$}
\end{equation}
Then we define
\[
 H^n_b(G,B(H)) = {\cl C}_n/{\cl B}_n
\]
where ${\cl B}_n \subset {\cl C}_n$ is defined by
\[
 {\cl B}_n= \{\partial F\mid F\in {\cl C}_{n-1}\}.
\]
In other terms, if $\ell_n=\ell_\infty(G^{n+1}; B(H))$, we have $\partial\partial=0$ on $\ell_n$ and
\begin{align*}
 {\cl C}_n &= \ker(\partial\colon \ \ell_n\to\ell_{n+1})\\
{\cl B}_n &\equiv \partial\ell_{n-1}\subset {\cl C}_n.
\end{align*}

For convenience we reserve the term [1]-cocycle for the elements of the space
 $C_1$. We will call cocycles the elements of $ {\cl C}_n$.

\begin{rem}\label{Dixrem1-1}
We have
\begin{equation}\label{Dixeq1-1}
 C_1/B_1\simeq {\cl C}_1/{\cl B}_1.
\end{equation}
Indeed, given a [1]-cocycle $f\in C_1$ we can define $F\in {\cl C}_1$ by setting
\[
 F(g_0,g_1) = g_0\cdot f(g^{-1}_0g_1).
\]
Indeed $F(gg_0,gg_1) = gF(g_0,g_1)$ is obvious and
\begin{align*}
 \partial F(g_0,g_1,g_2) &=F(g_1,g_2)-F(g_0,g_2)+F(g_0,g_1) \\
 &= g_1f(g^{-1}_1g_2) - g_0f(g^{-1}_0g_2) + g_0f(g^{-1}_0g_1)\\
&= g_0\cdot (s\cdot f(t) - f(st) + f(s))
\end{align*}
where $s = g^{-1}_0g_1$ and $t=g^{-1}_1g_2$. Thus $\partial F = 0$ iff $f$ is a [1]-cocycle. Moreover $f(g) = g\cdot T-T$ implies $F(g_0,g_1) = g_1\cdot T-g_0\cdot T$ and hence $F = \partial\varphi$ where $\varphi \in {\cl C}_0$ is defined by $\varphi(g) = g\cdot T$. 
Conversely, given $F\in {\cl C}_1$
the formula $f(t)=F(1,t)$ defines $f\in C_1$ (also note $F(1,1)=0$ and $F(1,t)=-F(t,1)$). Since the converse implications clearly hold, this proves \eqref{Dixeq1-1}. 
\end{rem}

\begin{rem}\label{Dixrem1-2}
A map $D\colon \ G\to B(H)$ is called a $\sigma$-derivation if
\begin{equation}
 D(st) = \sigma(s)D(t) + D(s)\sigma(t).\tag*{$\forall s,t\in G$}
\end{equation}
It is called inner if there is $T$ in $B(H)$ such that
\begin{equation}
 D(t) = \sigma(t)T - T\sigma(t).\tag*{$\forall t\in G$}
\end{equation}
The mapping
\[
f\to D_f
\]
defined by $D_f(t) = f(t)\sigma(t)$ is a 1-1 correspondence between bounded [1]-cocycles and bounded $\sigma$-derivations. Moreover, $f$ is trivial iff $D_f$ is inner. We leave the easy verification as an exercise for the reader.
\end{rem}

\begin{pro}\label{Dixpro1-1}

\item[\rm (i)] If $\Gamma={\bb F}_N$ (free group with $N$ generators) for $2\le N\le \infty$, then
with respect to the (left, say) regular representation of $\Gamma$ on $\ell_2\Gamma$ we have:
\[
 H^1_b(\Gamma, B(\ell_2\Gamma))\ne 0.
\]\\
\item[\rm (ii)] Any discrete group $G$ such that 
\[
H^1_b(G,B(H))\ne 0
\]
for some unitary representation $\pi\colon\ G\to B(H)$ is not unitarizable.

\end{pro}

\begin{proof}
(i) is well known, cf.\ e.g. \cite[p.~33-34]{P}. The proof of (ii) is classical:\ By Remark~\ref{Dixrem1-2}, we may assume that $G$ admits a bounded $\sigma$-derivation $D\colon \ G\to B(H)$ that is not inner.

Let then
\[
 \pi(g) = \begin{pmatrix}
           \sigma(g)&D(g)\\ 0&\sigma(g)
          \end{pmatrix}.
\]
Clearly $\pi$ is a uniformly bounded representation of $G$. It is known (cf.\ e.g.\ \cite[p.~80]{P}) that $\pi$ is unitarizable iff $D$ is inner. Thus $\pi$ (and hence $G$) is not unitarizable.
\end{proof}

\begin{lem}\label{Dixlem3-4}
Let $\Gamma$ be a group with a measure preserving action on a standard probability space $(X,\mu)$. Let $B = B(\ell_2\Gamma)$. We let $\Gamma$ act on $B$ by conjugation, i.e.
\begin{equation}
\gamma\cdot b= \lambda_\Gamma(\gamma)b \lambda_\Gamma(\gamma)^{-1}.\tag*{$\forall\gamma\in\Gamma 
\quad \forall b\in B$}
\end{equation}
We let $\Gamma$ act on $L_\infty(X,\mu;B)$ by setting
\begin{equation} 
(\gamma\cdot F)(x) =\gamma\cdot F(\gamma^{-1}x).\tag*{$\forall \gamma\in\Gamma \quad \forall F\in L_\infty(X,\mu;B)$}
\end{equation}
Then the embedding $B\subset L_\infty(X,\mu;B)$ mapping $B$ to constant functions defines an injective linear map
\begin{equation}\label{Dixeq3-5a}
H^1_b(\Gamma,B) \subset H^1_b(\Gamma,L_\infty(X,\mu;B))
\end{equation}
and similarly for $H^n_b$ any $n>1$.
\end{lem}

\begin{proof}
Let $f\colon \ \Gamma^2\to B$ be a bounded cocycle. Let $\hat f(\gamma_0,\gamma_1)\in L_\infty(X,\mu;B)$ be the constant function with value $f(\gamma_0,\gamma_1)$. Then $\hat f$ is clearly a bounded $L_\infty(X,\mu;B)$ valued cocycle that must be trivial if $f$ itself is trivial. Thus $f\to\hat f$ induces a linear mapping as in \eqref{Dixeq3-5a}. To prove that the latter mapping \eqref{Dixeq3-5a} is injective, it suffices to show that $\hat f$ is trivial in $L_\infty(X,\mu;B)$ only if $f$ is trivial in $B$.
To check this assume $\exists T\in L_\infty(X,\mu;B)$ such that 
\[
 \hat f(\gamma_0,\gamma_1) = \gamma_0\cdot T - \gamma_1\cdot T.
\]
This means for a.a.\ $x$
\[
 f(\gamma_0,\gamma_1) = \gamma_0\cdot T(\gamma^{-1}_0x) - \gamma_1T(\gamma^{-1}_1x).
\]
Let $\widehat T = \int T(x) d\mu(x)$. Since the $\Gamma$-action preserves $\mu$,   $\int T(\gamma^{-1}x) d\mu(x)$ does not depend on $\gamma$. Therefore we find
\[\
f(\gamma_0,\gamma_1) = \gamma_0\widehat T - \gamma_1\widehat T
\]
and since $\widehat T\in B$ this proves that $f$ is trivial as a $B$-valued cocycle.
\end{proof}

\begin{rem}\label{rem3-6} We will need the following elementary observation: let $(G_i)_{i\in I}$ be a family of groups.
Let $\cl G$ denote the direct sum of the family $(G_i)_{i\in I}$. An element of $\cl G$  can be viewed
as a family  $(g_i)_{i\in I}$ with $g_i\in G_i$ for all $i\in I$ such that $g_i$ differs from the unit element in
at most finitely many indices $i$. For each $i\in I$ let $\pi_i\colon\ G_i \to B(H)$ be a unitary representation.
Then, if for any $i\not= j$ in $I$ the ranges of $\pi_i$ and $\pi_j$ mutuallly commute, we can define
a unitary representation $\pi \colon\ {\cl G} \to B(H)$ extending each $\pi_i$ by setting simply
$$\pi ((g_i)_{i\in I} )=\prod_{i\in I} \pi_i(g_i).$$
The verification of this assertion is straightforward.
\end{rem}
\section{Free subgroups in non-amenable groups}\label{Dixsec2}

\indent 

It is classical that any group that contains ${\bb F}_2$ as a subgroup is non-amenable. A longstanding open problem (going back to the early days of amenability) was whether conversely any non-amenable group contains a copy of ${\bb F}_2$. This is disproved by A. Olshanskii in \cite{O}. More recently Olshanskii and Sapir \cite{OS} produced finitely generated and finitely presented non-amenable groups without free subgroups. Olshanskii's work used Adyan and Novikov's previous work on the fundamental Burnside problem. The Burnside group $B(m,p)$ is the universal group with $m$ generators such that $g^p=1$ for any element $g$ in the group. The celebrated Burnside problem asks whether there are $m,p$ for which $B(m,p)$ is infinite. Adyan and Novikov proved that $B(2,p)$ is infinite when $p\ge 665$ is odd. From work due to Ivanov and Lysionok (see \cite{Iv}) it is now known that $B(2,p)$ is infinite also when $p$ is a sufficiently large ($p\ge 8000$) \emph{even} number .

Shortly after Olshanskii's result circulated, Adyan \cite{A} proved that the groups $B(2,p)$ for $p$ odd and $p\ge 665$ are not amenable. Since they obviously do not contain ${\bb F}_2$ (actually not even ${\bb F}_1={\bb Z}$!), they provide examples of non-amenable groups without free subgroups. Despite these ``negative'' results, one could still hope to find a characterization of non-amenability by relaxing the notion of embedding of ${\bb F}_2$ into a group.

A recent outstanding result in that direction is the following due to Gaboriau and Lyons.

\begin{thm}[\cite{GL}]\label{Dixthm2-4}
A countable group $G$ is non-amenable iff it ``orbitally contains'' ${\bb F}_2$ in the sense of the following definition.
\end{thm}

\begin{defn}
Let $\Gamma,G$ be countable infinite groups. We say that $\Gamma$ ``orbitally embeds'' in $G$ and we write $\Gamma \overset{\circ}{\subset} G$ if there are free actions 
\[
 \Gamma \curvearrowright (X,\mu)\quad \text{and}\quad G\curvearrowright (X,\mu)
\]
on a standard probability space $(X,\mu)$ (i.e.\ a copy of the Lebesgue interval $([0,1], dt)$) such that for a.e.\ $x$
\[
  \Gamma x\subset Gx.
\]
\end{defn}
It follows obviously that almost any $G$-orbit is a union of $\Gamma$-orbits.

To get the flavor of this notion, a companion concept proposed by N.~Monod is appropriate:\ Let us denote by $[\Gamma,G]$ the set of all injective mappings $h\colon \ \Gamma\to G$ such that $h(1) = 1$. The group $\Gamma$ acts on this set $[\Gamma,G]$ as follows:
\[
 (\gamma\cdot h)(t) = h(t\gamma) h(\gamma)^{-1}\qquad \eqno \forall t,\gamma\in\Gamma.
\]

\begin{defn}\label{def}
We say that $\Gamma$ randomly embeds in $G$ and we write $\Gamma \overset{r}{\subset} G$ if there is a probability measure ${\bb P}$ on $[\Gamma,G]$ that is invariant under this action of $\Gamma$ on $[\Gamma,G]$.
\end{defn}

\begin{rk}
If ${\bb P}$ is supported by a singleton, i.e.\ ${\bb P} = \delta_h$ then ${\bb P}$ is $\Gamma$-invariant iff $h$ is a homomorphism, so in that case we find a classical embedding of $\Gamma$ into $G$. 
\end{rk}
\begin{pro}$\Gamma\overset{\circ}{\subset} G$ implies $\Gamma \overset{r}{\subset} G$.
\end{pro}
\begin{proof}
 Indeed, the fact that the actions of $\Gamma$ and $G$ on $(X,\mu)$ are free means that the mappings $\gamma\to \gamma x$ and $g\to gx$ are injective for a.e.\ $x$. Thus we can define
\[
 h_x\colon \ \Gamma\to G
\]
simply by setting $$\gamma x=h_x(\gamma)x.$$ Since $\Gamma x \subset Gx$ and the above maps are injective, $h_x(\gamma)$ is unambiguously defined for a.e.\ $x$. Clearly $h_x \in [\Gamma,G]$ for a.e. $x$.
 We have $h_x(\gamma\gamma_0)x=\gamma\gamma_0 x=h_{\gamma_0 x} (\gamma)\gamma_0 x=h_{\gamma_0 x} (\gamma)h_{ x} (\gamma_0)x$
 and hence $h_x(\gamma\gamma_0)=h_{\gamma_0 x} (\gamma) h_x(\gamma_0)$
 or equivalently $h_x(\gamma\gamma_0)h_x(\gamma_0)^{-1}=h_{\gamma_0 x} (\gamma) $. Therefore
 since
  $\Gamma$   acts by \emph{measure preserving} transformations : for any fixed $\gamma_0\in\Gamma$
\[
 (h_x(\gamma\gamma_0)h_x(\gamma_0)^{-1})_{\gamma\in\Gamma} \overset{\text{dist}}{=}\, (h_x(\gamma))_{\gamma\in\Gamma}.
\]
Equivalently, if ${\bb P}$ is the law on $[\Gamma,G]$ of the random element $h_x\in [\Gamma,G]$, then ${\bb P}$ is $\Gamma$-invariant so we have $\Gamma\overset{r}{\subset} G$. 
\end{proof}
 \begin{rk}   Note that  if  
 ${\bb P}$ is an invariant probability  on $[\Gamma,G]$ 
 then $\forall s,t\in \Gamma$ 
 we have
 $ h(s t^{-1})= (t^{-1}.h)(s) h(t^{-1})$  and also  (take $s=t$) $ 1=(t^{-1}.h)(t) h(t^{-1})$ and hence $h(t^{-1}))=(t^{-1}.h)(t)^{-1}$.
From this we deduce
 $$h(s t^{-1})= (u.h)(s)  (u.h)(t)^{-1}\quad {\rm with}\quad u=t^{-1}.$$ 
Thus we may write
\begin{equation}\label{Dixeqhom} h(s t^{-1}) \overset{\rm dist}{=} h(s) h(t)^{-1}\end{equation}
 \end{rk} 

\begin{pro}[\cite{M}]\label{Dixcor2-5+}
Assume $\Gamma \overset{r}{\subset} G$. If $G$ is amenable 
then $\Gamma$ is amenable too.
\end{pro}
\begin{proof} (Communicated by Stefaan Vaes) 
Let ${\bb P}$ be as in Definition \ref{def}. Let $f$ be a positive definite function on $G$.
We define a function $\hat f$ on $\Gamma$ by
$$\hat f(\gamma)= \int f(h(\gamma) ) d{\bb P}(h).$$
Then $\hat f$ is positive definite on $\Gamma$. Indeed,
let $x$ be any finitely supported function on $\Gamma$. We claim that
$$\sum \hat f(st^{-1}) x(s) \ovl {x(t)} \ge 0.$$
By \eqref{Dixeqhom} we have $\forall s,t\in \Gamma$ 
$$ \int   f(h(st^{-1}))   d{\bb P}(h)=\int f(h(s) h(t)^{-1})d{\bb P}(h)$$
and hence  since $f$ is positive definite we have
$$\sum \hat f(st^{-1}) x(s) \ovl {x(t)}=\int   f(h(st^{-1})) x(s) \ovl {x(t)} d{\bb P}(h)=\int   f(h(s) h(t)^{-1}) x(s) \ovl {x(t)} d{\bb P}(h)\ge 0.$$
Note that if $f$ is 
in ${\ell_1(G)}$, in particular if $f$ is finitely supported
then $\hat f \in \ell_1(\Gamma)$.
Indeed, let $\delta_g$ denote the indicator function of the singleton $\{g\}$.
Then
$\hat {\delta_g}(\gamma)= {\bb P}(\{h\mid  h(\gamma) =g\})$, and hence
$$\sum\nolimits_{\gamma\in \Gamma} \hat {\delta_g}(\gamma)=\int |h^{-1}(g)| d{\bb P}(h)\le 1,$$
and hence for any $f\in \ell_1(G)$ we have
$$\|\hat f\|_{\ell_1(\Gamma)}\le \sum |f(g)| \sum_{\gamma\in \Gamma} \hat {\delta_g}(\gamma)
\le \| f\|_{\ell_1(G)}.$$
Now, assuming $G$ amenable, let $(f_i)$ be a net of {finitely supported}
positive definite functions on $G$ tending  to the constant function 1 pointwise on $G$.
Since the sup norm of $f_i$ is attained at the unit, by dominated convergence,
$\hat f_i(\gamma)\to 1$ for any $\gamma$. Since $\hat f_i\in {\ell_1(\Gamma)}$, this
implies that $G$ is amenable by a well known criterion.
Indeed, this implies that the so-called Fourier algebra $A(G)$ has an approximate unit.
This property can also be easily reduced to the classical Kesten criterion:
since $\hat f_i $ is positive definite and in $A(G)$ it can be written
as $g_i * \tilde g_i$ where $\tilde g(t)=\ovl{g(t^{-1})}$ and $g_i\in \ell_2(\Gamma)$, and hence
from $\hat f_i(\gamma)\to 1$  
we  deduce by a well known argument that $\|g_i - \delta_{\gamma} * g_i\|_{\ell_2(\Gamma)} \to 0$
for any $\gamma$ in $\Gamma$. This equivalently means
that $\lambda_\Gamma$ has almost invariant vectors or weakly contains the trivial
representation. In other words we conclude that $\Gamma$ is amenable.
\end{proof}
\begin{rk} Note that we do not really use the full assumption that $\bb P$ is supported by injective
maps. It suffices that 
$$\sup\nolimits_{g\in G}   \int |h^{-1}(g)| d{\bb P}(h) <\infty.$$
In particular if $||h^{-1}(g)| \le C$ for all $g\in G$ for some $C>1$ we still conclude that
$\Gamma$ is amenable.
\end{rk}

Another proof can be derived from the following Proposition.

Let $\Omega$ denote the set of all maps $h\colon \ \Gamma\to G$ and let $\Omega_1 = \{h\in\Omega\mid h(1)=1\}$. We let $\Gamma$ act on $\Omega$ simply by right translations, i.e.\ we set
\begin{equation}\label{Dix6eq10}
 \forall\omega\in\Omega\quad \forall\gamma\in\Gamma \qquad\quad (\rho(\gamma)\omega)(t) = \omega(t\gamma) \qquad \forall t\in\Gamma.
\end{equation}
We let $G$ act on $\Omega$ by right multiplication, i.e.
\[
 \forall\omega\in\Omega\quad\forall g\in G\qquad\quad (g\cdot\omega)(t) = \omega(t)g^{-1}\qquad \forall t\in\Gamma.
\]
Let $m_G = \sum_{g\in G}\delta_g$ denote the counting (discrete Haar) measure on $G$. Then $\Omega = \cup_{g\in G} g\cdot\Omega_1$ (disjoint union), so that we may identify $\Omega$ with $G\times\Omega_1$. Let ${\bb P}$ be a measure on $\Omega_1$. We equip $\Omega$ with the measure $\mu$ corresponding to $m_G\times {\bb P}$.

\begin{pro}\label{proDix2}
If ${\bb P}$ is a probability supported by $[\Gamma,G]$ and  $\Gamma$-invariant, then $\mu$ is $\rho(\Gamma)$-invariant and the representation $\pi\colon \ \Gamma\to B(L_2(\mu))$ associated to the $\rho(\Gamma)$-action \eqref{Dix6eq10} is a unitary representation weakly equivalent to $\lambda_\Gamma$. Moreover, if $G$ is amenable, then $\pi$ admits almost invariant vectors (and hence $\Gamma$ is amenable).
\end{pro}

\begin{proof}
(Communicated by S. Vaes.) We first show that $\mu$ is invariant under $\rho(\gamma)$ $(\gamma\in\Gamma)$. For any $\omega\in\Omega$, say $\omega=g\cdot h$ with $h\in [\Gamma,G]\subset \Omega_1$ we have
\begin{equation}\label{Dix6eq11}
 \rho(\gamma)(g\cdot h) = (g(h(\gamma))^{-1})\cdot (\gamma\cdot h).
\end{equation}
(Recall $(\rho(\gamma)h)(t) = h(t\gamma) = (h(t\gamma)h(\gamma)^{-1})h(\gamma) = (\gamma\cdot h)(t) h(\gamma)$.) Therefore for any $F$ in $L_1(\mu)$ and $\gamma\in\Gamma$
\begin{align*}
 \int F(\rho(\gamma)\omega) d\mu(\omega) &= \sum\nolimits_g \int F(\rho(\gamma)(g\cdot\omega))\  d{\bb P}(\omega) = \sum\nolimits_{g\in G} \int F(g h(\gamma)^{-1}\cdot (\gamma\cdot\omega)) \ d{\bb P}(\omega)\\
&= \sum\nolimits_{g'\in G} \int F(g'\cdot(\gamma\cdot \omega)) d{\bb P}(\omega) = \sum\nolimits_{g'\in G} \int F(g'\cdot\omega) \ d{\bb P}(\omega) = \int F\ d\mu.
\end{align*}
Therefore $\rho(\gamma)$ preserves $\mu$ and $\pi$ (defined by $\pi(\gamma)F(\omega) = F(\rho(\gamma)^{-1} \cdot\omega))$ is unitary on $L_2(\mu)$.

To show that $\pi$ is weakly equivalent to $\lambda_\Gamma$, it suffices to show that the action \eqref{Dix6eq10} admits a fundamental domain $Y\subset\Omega$, i.e.\ a (measurable) subset such that the sets $\{\rho(\gamma)^{-1}\cdot Y\mid \gamma\in\Gamma\}$ form a disjoint covering of $\Omega$ (up to $\mu$-negligible sets). Indeed, assuming that this holds we may identify $\Omega$ with $Y\times\Gamma$ so that $L_2(\mu) \simeq L_2(Y,\mu_{|Y})\otimes \ell_2(\Gamma)$ and $\pi \simeq Id\otimes\rho_\Gamma$. Thus $\pi$ is weakly equivalent to $\rho_\Gamma$ or $\lambda_\Gamma$.

We will now show that such a $Y$ exists. Let $\{g_1,g_2,\ldots\}$ be an enumeration of $G$. Let $A_n = \{h\in \Omega\mid g_n\in h(\Gamma)\}$. Note that $\Omega = \cup A_n$. Let $A'_1 = A_1$ and
\[
 A'_n = A_n \backslash (A_1\cup\cdots\cup A_{n-1}).
\]
Let
\begin{equation}
 A'_n(\gamma) = \{h\mid h(\gamma) = g_n\} \backslash (A_1\cup\cdots\cup A_{n-1}).\tag*{$\forall \gamma\in\Gamma$}
\end{equation}
Note that the sets $A_n$ are invariant under $\rho(\gamma)$, and for each fixed $n$ the sets $\{A'_n(\gamma) \mid \gamma\in\Gamma\}$ are essentially disjoint since $\mu$ is supported by injective maps. Moreover we have
\[
 \rho(\gamma)^{-1} A'_n(1) = A'_n(\gamma).
\]
Thus $A'_n(1)$ is a fundamental domain for the $\rho$-action restricted to $A'_n$. It is now easy to check that $Y = \bigcup_{n\ge 1} A'_n(1)$ is the desired fundamental domain.

 Lastly, if $G$ is amenable, there is a net $(\xi_\alpha)$ in the unit sphere of $\ell_2(G)$ such that $\|\rho_G(g)\xi_\alpha-\xi_\alpha\|_2\to 0$ for any $g$ in $G$. Let then $\forall\omega_1\in \Omega_1$ $\forall g\in G$
\[
 F_\alpha(g\cdot\omega_1) = \xi_\alpha(g).
\]
Since ${\bb P}$ is a probability,  $F_\alpha$ is in the unit sphere of $L_2(\mu)$ and since by \eqref{Dix6eq11} 
\[
 F_\alpha(\rho(\gamma)(g\cdot\omega_1)) = F_\alpha((g\omega_1(\gamma)^{-1})\cdot (\gamma\cdot h)) = \xi_\alpha  (g\omega_1(\gamma)^{-1})
\]
 we have $\|\pi(\gamma^{-1})F_\alpha-F_\alpha\|_{L_2(\mu)}\to 0$ $\forall \gamma\in \Gamma$. Since $\pi \simeq Id\otimes\rho_\Gamma$, $\rho_\Gamma$ itself admits almost invariant vectors and hence $\Gamma$ is amenable.
\end{proof}

Monod mentions in \cite{M} that, as a corollary of Theorem~\ref{Dixthm2-4}
and   Proposition \ref{Dixcor2-5+}, we have:

\begin{cor}\label{Dixcor2-5}
An infinite group $G$ is non-amenable iff ${\bb F}_2 \overset{r}{\subset} G$.
\end{cor}
\begin{rk}The sequel implicitly contains another proof that $G$ is not-amenable if ${\bb F}_2 \overset{o}{\subset} G$.
Indeed, we will show that, if ${\bb F}_2 \overset{o}{\subset} G$,  $A\wr G={A}^{(G)}\rtimes G$ is not unitarizable, 
and hence not amenable,
with (say) $A=\bb Z$. A fortiori, since ${\bb Z}^{(G)}$ is amenable, $G$ is not-amenable.
\end{rk}

\section{The Monod--Shalom induction}\label{Dixsec3}

\noindent 

Let us assume $\Gamma \overset{\circ}{\subset} G$, so that we have actions of $\Gamma$ and $G$ so that $\Gamma x\subset Gx$ for a.e.\ $x$. Let ${\cl R}\subset X\times X$ be the measurable equivalence relation associated to the $G$-action, i.e.
\[
 {\cl R} = \{(s,x)\mid s\in Gx, \ x\in X\}.
\]
We equip ${\cl R}$ with a $\sigma$-finite measure $M$ defined (for suitably restricted $F$, say for $F$ bounded, measurable and supported by a finite union of equivalence classes) by
\[
 \int F~dM = \sum\nolimits_{g\in G} \int F(gx,x)d\mu(x).
\]
We will let $G$ act on ${\cl R}$ on the left so that
\begin{equation}
 (g\cdot F) (s,x) = F(g^{-1}s,x)\tag*{$\forall(s,x)\in {\cl R}\ \ \forall g\in G$}
\end{equation}
and we let $\Gamma$ act on the right so that
\begin{equation}
 (\gamma\cdot F) (s,x) = F(s,\gamma^{-1}x).\tag*{$\forall \gamma\in\Gamma$}
\end{equation}\\ 
Then these actions obviously commute, and since both $G\curvearrowright X$ and $\Gamma \curvearrowright X$ preserve the measure $\mu$, the actions just defined on $F$ extend to two commuting unitary representations of $G$ and $\Gamma$ on $L_2({\cl R},M)$. Let us denote $L_2({\cl R}) = L_2({\cl R},M)$. \\ The $G$-action on ${\cl R}$ obviously admits the diagonal $\Delta = \{(x,x)\mid x\in X\} \subset X\times X$ as a fundamental domain:\ Indeed we have a 1-1 correspondence $Gx\times\{x\}\leftrightarrow (x,x)\in\Delta$, between the orbits of the $G$-actions on ${\cl R}$ and the points of $\Delta$.

Although it is a bit less obvious, the $\Gamma$-action also admits a fundamental domain $Y\subset X\times X$ that can be chosen so that $\Delta\subset Y$. One way to see this is to enumerate the group $G$ as $\{g_n\mid n\ge 0\}$. We set $g_0=1$ for convenience. We then define recursively for any $x$ in $X$
\begin{align*}
 N_0(x) &= 0,\quad N_1(x) = \min\{n\mid g_nx\notin \Gamma x\},\ldots\\
N_k(x) &= \min\{n\mid g_nx\notin \Gamma x\cup \Gamma g_{N_1}x \cup\cdots\cup g_{N_{k-1}}x\}.
\end{align*}
We then set $Y = \{(x,g_{N_k(x)}x)\mid k\ge 0\}$. Clearly our choices for $N_k$ are measurable. Consider a $\Gamma$-orbit $\Gamma(s,x) = s\times \Gamma x$. Note that $\Gamma x\subset Gx$, and $x\in Gs$ so $\Gamma x\subset Gs$ can be clearly identified with $\Gamma g_{N_k(s)}s$ for some $k$. This shows that the correspondence $(s,g_{N_k(s)}s) \leftrightarrow \Gamma(s,x)$ is bijective (up to negligible sets). In other words, $Y$ is a fundamental domain for the $\Gamma$-action. Thus we have a measure isomorphism
\[
 Y\times\Gamma \longrightarrow {\cl R}
\]
of the form $(y,\gamma) \longrightarrow \gamma^{-1}y$ $(\gamma\in\Gamma,y\in Y)$ that gives us new coordinates to represent a point of ${\cl R}$ by a pair $(y,\gamma)$.

Note that this is a measure preserving isomorphism between $({\cl R},M)$ and $Y\times\Gamma$ equipped with the product of the restriction $M_{|Y}$ of $M$ to $Y$ with the counting measure on $\Gamma$. In these new coordinates, the action of $\Gamma$ is easy to describe:\ for any $t$ in $\Gamma$, $t^{-1}(y,\gamma) = t^{-1}\gamma^{-1} y = (y,\gamma t)$. Thus $t$ acts just like right translation on the $\Gamma$ coordinate. More precisely, for any $F\in L_2({\cl R}) = L_2(Y\times\Gamma)$ we have
\begin{equation}
 (t\cdot F) (y,\gamma) = F(t^{-1}(y,\gamma))= F(t^{-1}\gamma^{-1}y)
  = (id\otimes \rho_\Gamma(t))(F) (y,\gamma) \tag*{$\forall t\in \Gamma$}
\end{equation}
where $\rho_\Gamma\colon \ \Gamma\to B(\ell_2\Gamma)$ is the right regular representation. 

To describe the action of $G$ with the new coordinates, we define $\alpha(g,y)\in \Gamma$ and $g\cdot y\in Y$ by the identity
\begin{equation}\label{Dixeq3-2}
 gy = (g\cdot y,\alpha(g,y))
\end{equation}
or equivalently $gy = \alpha(g,y)^{-1}(g\cdot y)$. Warning: \ The reader should distinguish the original $G$-action $gy$ from $g\cdot y$ that we just defined.

Since the $G$- and $\Gamma$-actions commute, we have
\begin{gather*}
 g(y,\gamma) = g\gamma^{-1}y = \gamma^{-1}gy = \gamma^{-1}\alpha(g,y)^{-1}(g\cdot y)
 = (g\cdot y,\alpha(g,y)\gamma).
\end{gather*}
It is  important to observe that the first component $g\cdot y$ of $g(y,\gamma)$ does not depend on $\gamma$, the  second component of $g=(y,\gamma)$

This shows that the subspace $L^\infty(Y,B(\ell_2\Gamma))\subset B(L_2(Y\times \Gamma))$ formed of multiplication operators by a function $\varphi\colon \ Y\to B(\ell_2\Gamma)$ is stable by the action of $G$. Indeed, denoting by ${\cl M}_\varphi$ this multiplication operator acting on $L_2(Y\times\Gamma)$ we have 
\begin{equation} \label{Dixeq3-3}
 \forall g\in G\quad \pi(g) {\cl M}_\varphi\pi(g)^{-1} = {\cl M}_{g\cdot \varphi} {}
\end{equation}
where 
\begin{equation}\label{Dixeq3-4}
(g\cdot \varphi) (y) \overset{\text{def}}{=} \alpha(g^{-1},y)^{-1} \cdot \varphi(g^{-1}\cdot y).
\end{equation}  
More explicitly, if we set $a=\alpha(g^{-1},y)^{-1}$ we have  $(g\cdot \varphi) (y)=\lambda_{\Gamma} (a ) \varphi(g^{-1}\cdot y)\lambda_{\Gamma} (a )^{-1}$.
 
In particular, this shows that the subspace $L^\infty(Y)\otimes 1 \subset B(L_2(Y\times \Gamma))$ formed of multiplications by a function depending only on $y$ is stable under the action of $G$:\ Indeed, if $\varphi(y) = \Phi(y)Id_{\ell_2(\Gamma)}$, we have
\begin{equation}\label{Dixeq3-5}
 \pi(g) {\cl M}_\varphi\pi(g^{-1}) = {\cl M}_{\widetilde \varphi}
\end{equation}
where $\widetilde \varphi(y) = \Phi(g^{-1}\cdot y) Id_{\ell_2\Gamma}$. \\ 
In the next statement, we set $B = B(L_2({\cl R}))$ but actually $B$ can be any dual Banach space that is both a $G$-module and a $\Gamma$-module with commuting actions (i.e.\ $B$ is a $(G\times\Gamma)$-module).

We consider the space ${\cl L} = L_\infty({\cl R};B)$ of weak-$*$ measurable essentially bounded functions with the usual norm. We equip ${\cl L}$ with a $G$-module action by setting
\begin{equation}
 (g\cdot F)(\omega) = g\cdot F(g^{-1}\omega),\tag*{$\forall F\in {\cl L}$}
\end{equation}
and similarly for the $\Gamma$-action. We now formulate   Monod and Shalom's ``induction in cohomology":

\begin{pro}\label{Dixpro3-5}
Let ${\cl L} = L_\infty({\cl R};B)$ as above, let ${\cl L}^\Gamma\subset {\cl L}$ be the submodule formed of the $\Gamma$-invariant elements, i.e.\ those $F$ such that
\begin{equation}
 \gamma\cdot F(\gamma^{-1}\omega) = F(\omega),\tag*{$\forall \gamma\in\Gamma \ \ \forall\omega \in {\cl R}$}
\end{equation}
and similarly for ${\cl L}^G\subset {\cl L}$. Note that (since the actions commute) ${\cl L}^\Gamma$ (resp.\ ${\cl L}^G$) is a $G$-module (resp.\ $\Gamma$-module). We have then an embedding (i.e.\ an injective linear map)
\begin{equation}\label{Dixeq3-10}
 H^1_b(\Gamma,{\cl L}^G) \subset H^1_b(G,{\cl L}^\Gamma)
\end{equation}
and actually for any $n\ge 1$, $H^n_b(\Gamma,{\cl L}^G) \subset H^n_b(G,{\cl L}^\Gamma)$.
\end{pro}

\begin{proof}
For each $\omega\in{\cl R}$, we have a map $\chi_\omega\colon \ G\to\Gamma$ defined by $\chi_\omega(g)=\gamma$ if $g\omega \in \gamma Y$ (i.e.\ $g\omega = (y,\gamma^{-1})$ for some $y$ in $Y$). Similarly we have a map $\kappa_\omega\colon \ \Gamma\to G$ defined by $\kappa_\omega(\gamma) = t$ if $\gamma\omega \in t\Delta$, i.e. $\gamma \omega = (tx,x)$ for some $x\in X$. We claim that 

\begin{equation}\label{Dixeq3-11}\chi_\omega(\kappa_\omega(\gamma^{-1})^{-1})=\gamma.\end{equation}

Indeed, $\kappa_\omega(\gamma^{-1})=g^{-1}$ means that $\gamma^{-1}\omega \in g^{-1}\Delta$ and hence (since the actions commute) $g\omega\in \gamma\Delta$ but since $\Delta\subset Y$, we have $g\omega \in \gamma Y$ and hence $\chi_\omega(g)=\gamma$.

The embedding \eqref{Dixeq3-10} is then defined as follows:\ for any $f\colon \ \Gamma^2\to {\cl L}^G$ we define $\chi^*f \colon \ G^2\longrightarrow {\cl L}^\Gamma$ as follows:
\[
 \chi^*f(g_0,g_1)(\omega) = f(\chi_\omega(g^{-1}_0), \chi_\omega(g^{-1}_1))(\omega).
\]
Similar for any $\varphi\colon \ G^2\to {\cl L}^\Gamma$ we define $\kappa^*\varphi\colon \ \Gamma^2\to {\cl L}^G$ by:
\[
 \kappa^*\varphi(\gamma_0,\gamma_1)(\omega) = \varphi(\kappa_\omega(\gamma^{-1}_0), \kappa_\omega(\gamma^{-1}_1))(\omega).
\]
Note that since the groups are countable, $\omega \mapsto \chi_\omega(g^{-1})$ is measurable and ${\it countably \ valued}$ so it is easy to verify measurability of $ \chi^*f(g_0,g_1)$ and $ \kappa^*\varphi(\gamma_0,\gamma_1)$.\\
 Let us now check that $\chi^*f$ defines a cocycle. Note that $\chi_\omega((gg_j)^{-1}) = \gamma_j$ means $g^{-1}_jg^{-1}\omega\in \gamma_jY \Leftrightarrow g^{-1}_j(g^{-1}\omega)\in \gamma_jY$ and hence $\chi_\omega((gg_j)^{-1}) = \chi_{g^{-1}\omega}(g^{-1}_j)$. By definition of ${\cl L}^\Gamma$, this implies (since $f(\cdot,\cdot)\in {\cl L}^\Gamma$):
\begin{align*}
 \chi^*f(gg_0,gg_1) &= f(\chi_{g^{-1}\omega}(g^{-1}_0), \chi_{g^{-1}\omega}(g^{-1}_1))(\omega)\\
&= g\cdot f(\chi_{g^{-1}\omega}(g^{-1}_0), \chi_{g^{-1}\omega}(g^{-1}_1)) (g^{-1}\omega)\\
&= (g\cdot \chi^*f)(g_0,g_1).
\end{align*}
Let $\varphi = \chi^*f$. The verification that $\partial\varphi=0$ is immediate from $\partial  f = 0$ and the definition of $\chi^*f$. Actually, the definition of $\chi^*$ can be extended to functions on $G^{n+1}$ for any $n\ge 1$ and we have obviously $\chi^*\partial = \partial\chi^*$. Since $\chi^*$ obviously preserves boundedness, this shows that $\chi^*$ defines a linear map as required by \eqref{Dixeq3-10}. Lastly by \eqref{Dixeq3-11}, we have $\kappa^*\chi^*f=f$ so that $\chi^*$ must be injective. The proof for any $n>1$ is identical.
\end{proof}

\section{Main result}\label{Dixsec4}

\begin{thm}\label{Dixthm4-10}
If $\Gamma = {\bb F}_2$ orbitally embeds in $G$ then for any infinite Abelian group $A$, the wreath product $A\wr G$ is not unitarizable.
\end{thm}

\begin{proof}
By Proposition~\ref{Dixpro1-1}, it suffices to show that $A\wr G$ admits a unitary representation $\widehat\pi \colon \ A\wr G\longrightarrow B(H)$ for which there is a bounded $\widehat\pi$-derivation that is not inner. We use the preceding notation for $X,{\cl R}$ and $Y$. We start from (cf.\ Proposition~\ref{Dixpro1-1}(i)) a non-inner bounded derivation $D\colon \ \Gamma\longrightarrow B(\ell_2\Gamma)$. Let $B = B(\ell_2\Gamma)$. By Lemma~\ref{Dixlem3-4} and Remark \ref{Dixrem1-2}, we have a non-trivial [1]-cocycle $\hat f\colon \ \Gamma\longrightarrow L^\infty(X,\mu;B)$, defined  by $\hat f(\gamma) = D(\gamma)\lambda(\gamma)^{-1}$ where $\lambda=\lambda_\Gamma$. (Actually, $  D$ itself, when viewed as a derivation into $L^\infty(X,\mu;B)$, is not inner.)
We equip $B$ with the $\Gamma$-module structure associated to the action of $\Gamma$ by conjugation, i.e.
\begin{equation}
 \gamma\cdot b = \lambda(\gamma)b\lambda(\gamma)^{-1}.\tag*{$\forall\gamma\in\Gamma\quad \forall b\in B$}
\end{equation}
Thus ${\cl L} = L_\infty({\cl R};B)$ is a $\Gamma$-module defined by:
\begin{equation}\label{Dixeq4-7}
\forall\gamma\in\Gamma\quad \forall F\in{\cl L}  \qquad (\gamma F)(\omega) = \gamma\cdot F(\gamma^{-1}\omega).
\end{equation}
 Note that $\varphi\in {\cl L}^\Gamma$ is characterized by
\begin{equation}\label{Dixeq4-10}
\forall\gamma\in\Gamma  \qquad \varphi(\omega)=\gamma\cdot\varphi(\gamma^{-1}\omega).
\end{equation}
Of course, the preceding two equalities are meant to hold   for a.a.  $\omega \in{\cl R}$.\\
As a $G$-module we equip $B$ with the trivial action, i.e.
\begin{equation}
 gb=b.\tag*{$\forall b\in B\quad \forall g\in G$}
\end{equation}
Then ${\cl L} = L_\infty({\cl R};B)$ is a $G$-module with respect to the action defined by
\begin{equation}\label{Dixeq4-8}
\forall F\in {\cl L}\quad \forall g\in G\qquad\qquad (g\cdot F)(\omega) = F(g^{-1}\omega).
\end{equation}
Note that $F\in {\cl L}^G$ iff we have
\begin{equation}\label{Dixeq4-11}
 \forall g\in  G\qquad\qquad F(g^{-1}\omega) = F(\omega)\qquad \forall \omega\in {\cl R}.
\end{equation}
We claim that, as $\Gamma$-modules, we have
\begin{equation}\label{Dixeq4-9}
 L_\infty(X,\mu;B) \simeq L_\infty({\cl R},B)^G.
\end{equation}
Indeed, for any $F\in L_\infty({\cl R},B)^G$ we have
\[
 F(x,x)=F(s,x) \qquad \forall(s,x)\in {\cl R}
\]
since $(s,x)$   belongs to $G\cdot (x,x)$, and, by \eqref{Dixeq4-8}, any $F\in {\cl L}^G$ is constant on the $G$-orbits.
Let $\widetilde F(x) = F(x,x)$. Clearly $F\to \widetilde F$ is an isometric $\Gamma$-module isomorphism from ${\cl L}^G$  onto $L_\infty(X,\mu;B)$ (to check this note that for any $\gamma\in \Gamma$
 we have $F(x,\gamma^{-1}x)=F(\gamma^{-1}x,\gamma^{-1}x)$ since $x$ and $\gamma^{-1}x$ are in the same $G$-orbit, namely $Gx$).

From our claim \eqref{Dixeq4-9}, we deduce that $H^1_b(\Gamma, {\cl L}^G)\ne \{0\}$. By Proposition~\ref{Dixpro3-5}, we deduce that $H^1_b(G,{\cl L}^\Gamma)\ne 0$. We now claim that, as $G$-modules, we have
\begin{equation}\label{Dixeq4-12}
 {\cl L}^\Gamma \simeq L_\infty(Y;B)
\end{equation}
where  $L_\infty(Y;B)$ is equipped with the specific $G$-module structure defined in \eqref{Dixeq3-3}. Indeed, let $\varphi\in {\cl L}^\Gamma$. Then using the ``new coordinates'' we have ${\cl R} \simeq Y\times \Gamma$ and with these the $\Gamma$-invariance of $\varphi$ yields
\[
 \varphi(y,\gamma) = \gamma^{-1}\varphi(y,1).
\]
Define then $F_\varphi \in L_\infty(Y,B)$ by $F_\varphi(y) = \varphi(y,1)$. A simple calculation shows that
\[
 (g\cdot\varphi)(y,\gamma) = \varphi(g^{-1}(y,\gamma)) = \varphi(g^{-1}\gamma^{-1}y) = \varphi(\gamma^{-1}g^{-1}y) = \gamma^{-1}\varphi(g^{-1}y)
\]
where the last two equalities follow respectively from the commutation of the $G$- and $\Gamma$-actions and from \eqref{Dixeq4-10}. But then, by \eqref{Dixeq3-2} we have
\[
 g^{-1}y = \alpha(g^{-1},y)^{-1} (g^{-1}\cdot y)
\]
and hence we find by \eqref{Dixeq4-10} again
\[
 (g\cdot\varphi)(y,\gamma) = (\gamma^{-1}\alpha(g^{-1},y)^{-1})\cdot \varphi(g^{-1}\cdot y).
\]
This means that $F_{g\cdot\varphi}(y) = \alpha(g^{-1},y)^{-1}\cdot \varphi(g^{-1}\cdot y)$. Equivalently, we find by \eqref{Dixeq3-2} $F_{g\cdot\varphi}= g\cdot F_{\varphi}$, so that the correspondence $\varphi \to F_\varphi$ is a $G$-module morphism. Since it is clearly invertible (using $\varphi(y,\gamma) = \gamma^{-1} F_\varphi(y)$) this proves the claim \eqref{Dixeq4-12}.

At this stage, we now know that $H^1_b(G,L_\infty(Y;B))\ne \{0\}$. In other words, there is a non-trivial bounded cocycle $f\colon \ G^2\longrightarrow L_\infty(Y;B)$ where $L_\infty(Y;B)$ is equipped with the $G$-module structure \eqref{Dixeq3-2}.

Consider the embedding $L_\infty(Y;B) \subset B(L_2(Y\times \Gamma))$, taking $\varphi\in L_\infty(Y,B)$ to the (operator valued) multiplication ${\cl M}_\varphi$. Recall that by \eqref{Dixeq3-3} we have
\begin{equation}\label{Dixeq4-14}
 \forall g\in G\qquad\qquad \pi(g) {\cl M}_\varphi\pi(g)^{-1} = {\cl M}_{g\cdot \varphi}.
\end{equation}
In other words, we may identify $L_\infty(Y,B)$ with the submodule of $B(L_2(Y\times\Gamma))$ consisting of the operators ${\cl M}_\varphi$ equipped with the $G$-action defined by restricting the action defined on $B(L_2(Y\times\Gamma))$ by
\begin{equation}\label{Dixeq4-15}
\forall g\in G\quad \forall T\in B(L_2(Y\times\Gamma))\qquad\qquad g\cdot T = \pi(g) T\pi(g)^{-1}.
\end{equation}
We are now close to the situation described in (ii) in Proposition~\ref{Dixpro1-1}. However, there is one last problem:\ we do obtain a bounded [1]-cocycle $f\colon \ G\longrightarrow B(L_2(Y\times \Gamma))$ with $B(L_2(Y\times\Gamma))$ equipped with \eqref{Dixeq4-15}, but we only know that $f$ is non-trivial into $L_\infty(Y,B)$. This is where the extension to the larger group $A\wr G$ is useful.

First we extend $\pi$ to a unitary representation $\widehat\pi\colon \ A\wr G\longrightarrow B(L_2(Y\times\Gamma))$ in such a way that $\widehat\pi(A^{(G)}) =L_\infty(Y)\otimes I$. To do this we note that the von~Neumann algebra of $A$ can be identified (as a von~Neumann) algebra with $L_\infty(Y)\otimes I$. This yields a unitary representation $\sigma\colon \ A\to L_\infty(Y)\otimes I$ with range generating $L_\infty(Y) \otimes I$ (as a von~Neumann algebra) i.e. such that $\sigma(A)'' = L_\infty(Y)\otimes I$. We now recall that by \eqref{Dixeq3-3} we have
\[
 \pi(g)(L_\infty(Y)\otimes 1) \pi(g)^{-1} \subset L_\infty(Y)\otimes 1.
\]
Therefore the family of representations $\pi_g\colon\ A\to L_\infty(Y)\otimes I \subset B(L_2(Y\times \Gamma))$
($g\in G)$, defined by $\pi_g(a)= \pi(g)\sigma(a) \pi(g)^{-1}$ for any $a\in A$, have mutually commuting ranges.
By Remark \ref{rem3-6}, the product of this family of representations defines a unitary
representation $\widehat\sigma$ on the product group $A^{(G)}$.
Thus we have extended $\sigma$ to a unitary representation
\[
\widehat\sigma\colon \ A^{(G)} \longrightarrow L_\infty(Y)\otimes 1
\]
such  that
\begin{equation}\label{Dixeq4-16}
 \forall a\in A^{(G)}\qquad\qquad \widehat\sigma(g\cdot a) = \pi(g) \widehat\sigma(a) \pi(g)^{-1},
\end{equation}
where we recall that $g\cdot a$ is the ``shift action'' on $A^{(G)}$ defined by $(g\cdot a)(t) = a(g^{-1}t)$.

From \eqref{Dixeq4-16} it follows that the mapping $A^{(G)}\times G\longrightarrow B(L_2(Y\times \Gamma))$ taking $(a,g)$ to $\widehat\sigma(a)\pi(g)$ defines a unitary representation $\widehat\pi$ on $A\wr G$ extending $\pi$. Of course, this allows us to extend the action of $G$ on $B(L_2(Y\times \Gamma))$ to an action of $A\wr G$ by setting
\begin{equation}
 \theta\cdot T = \widehat\pi(\theta) T\widehat\pi(\theta)^{-1}.\tag*{$\forall \theta\in A\wr G\quad \forall T\in B(L_2(Y\times \Gamma))$}
\end{equation}
Note that since $\widehat\pi(A^{(G)})\subset L_\infty(Y)\otimes I$ we have
\begin{equation}\label{Dixeq4-17}
 \forall\theta\in A^{(G)}\quad \forall T\in L_\infty(Y,B)\qquad\qquad \theta\cdot T = T.
\end{equation}
We extend our (non-trivial) [1]-cocycle $f\colon \ G\longrightarrow L_\infty(Y,B)$ to a [1]-cocycle $\hat f\colon\ A\wr G\longrightarrow B(L_2(Y\times \Gamma))$ by setting: \ $\forall a\in A^{(G)}$ $\forall g\in G$
\[
 \hat f(ag) = f(g).
\]
Note that, since $f(g)\in L_\infty(Y,B)$  $(g\in G)$, \eqref{Dixeq4-17} implies:
\begin{equation}
 a\cdot f(g) = f(g)\qquad \forall g\in G.\tag*{$\forall a\in A^{(G)}\subset A\wr G$}
\end{equation}
Using this it is easy to check that $\hat f$ is a [1]-cocycle ($\hat f(aga'g') = f((aga'g^{-1})(gg')) = f(gg') = g\cdot f(g') + f(g) = (ag)\cdot f(g') +f(g)$ where \eqref{Dixeq4-17} is used again at the last step).

Our last claim is that the bounded [1]-cocycle $\hat f\colon \ A\wr G\longrightarrow B(L_2(Y\times \Gamma))$ is non-trivial (now viewed as a [1]-cocycle into $B(L_2(Y\times\Gamma))$). Indeed, assume that there is $T$ in $B(L_2(Y\times\Gamma))$ such that 
\begin{equation}
 \hat f(\theta) = \theta\cdot T-T.\tag*{$\forall\theta\in A\wr G$}
\end{equation}
We have then $f(g) = g\cdot T-T$ for any $g\in G$,  but also since $f$ is a [1]-cocycle we know that $f(1) =0$. This gives us
\begin{equation}
 \hat f(a)=0\tag*{$\forall a\in A^{(G)}$}
\end{equation}
and hence $a\cdot T = T$\quad $\forall a\in A^{(G)}$. Equivalently $\widehat\pi(a)T \widehat\pi(a)^{-1}=T$. But since $\widehat\pi(A^{(G)})$ fills $L_\infty(Y)\otimes 1$ we conclude that $T\in (L_\infty(Y)\otimes I)'$ and hence that $T\in L_\infty(Y;B)$. But this would mean that $f$ is trivial as a [1]-cocycle into $L_\infty(Y;B)$. This contradiction completes the proof of our claim that $\hat f$ is non-trivial. By the second part of Proposition~\ref{Dixpro1-1}, we conclude that $A\wr G$ is not unitarizable.
\end{proof}

\begin{rk}
Suppose $D\colon \ \Gamma\longrightarrow B(\ell_2\Gamma)$ is the original non-inner bounded $\lambda$-derivation on $\Gamma$. It yields a non-trivial [1]-cocycle $f\colon \ \Gamma\to B$ defined by $f(\gamma) = D(\gamma) \lambda(\gamma)^{-1}$. We associate to it a non-trivial $F\colon \ \Gamma^2\to B$ by setting $F(\gamma_0,\gamma_1) = \gamma_0f(\gamma^{-1}_0\gamma_1)$. We then view $F$ as a cocycle $F\colon \ \Gamma^2\longrightarrow L_\infty(X;B)$, viewing $F(\gamma_0,\gamma_1)$ as a constant function on $X$. We then define $\check{F}\colon \ \Gamma^2 \longrightarrow L_\infty({\cl R},B)^G$ by setting $\check{F}(\gamma_0,\gamma_1)(t,x) = F(\gamma_0,\gamma_1)(x) = F(\gamma_0,\gamma_1)$. We then associate to $\check{F}$ the cocycle on $G^2$ defined by
\[
(\chi^*\check{F})(g_0,g_1)(\omega) = \check{F}(\chi_\omega g^{-1}_0,\chi_\omega g^{-1}_1)(\omega).
\]
If $\omega = \gamma^{-1}y$ $(\gamma\in \Gamma, y\in Y)$ we have $g^{-1}_j\omega = (\gamma^{-1}\alpha(g^{-1}_j,y)^{-1})(g^{-1}_j\cdot y)$ with $g^{-1}_jy\in Y$ and hence $\chi_\omega g^{-1}_j = \gamma^{-1} \alpha(g^{-1}_j,y)^{-1}$. Thus we find
\begin{align*}
 (\chi^*\check{F})(g_0,g_1)(y,\gamma) &= \gamma^{-1}\check{F}(\alpha(g^{-1}_0,y)^{-1}, \alpha(g^{-1}_1,y)^{-1})(\omega)\\
&= \gamma^{-1}F(\alpha(g^{-1}_0,y)^{-1}, \alpha(g^{-1}_1,y)^{-1}).
\end{align*}
Let
\[
 E(g_0,g_1)(y) \overset{\text{def}}{=} F(\alpha(g^{-1}_0,y)^{-1}, \alpha(g^{-1}_1,y)^{-1})
\]
$E$ is a non-trivial cocycle into $L_\infty(Y,B)\subset B(L_2(Y\times\Gamma))$. Thus $g\longrightarrow E(1,g)$ is a non-trivial [1]-cocycle. We extend it on $A\wr G$ by setting
\[
 \widehat E(ag) = E(g).
\]
Then we finally obtain a non-inner derivation $\widetilde D\colon \ A\wr G\longrightarrow B(L_2({\cl R}))$ by setting
\begin{align*}
 \widetilde D(ag) &= \widetilde E(ag) \widehat\pi(ag)\\
&= E(g) \widehat\pi(ag)\\
&= F(1, \alpha(g^{-1},y)^{-1})\widehat \pi(ag)\\
&= \widehat\pi(a) F(1,\alpha(g^{-1},y)^{-1}) \widehat\pi(g).
\end{align*}
Using $F(1,\gamma) = f(\gamma) = D(\gamma) \lambda(\gamma)^{-1}$ we obtain finally
\[
 \widetilde D(g) = D(\alpha(g^{-1},y)^{-1}) \lambda(\alpha(g^{-1},y)) \pi(g).
\]
Consider $\xi \in L_2({\cl R})=L_2(Y\times \Gamma)$ of the form $\xi(y,\gamma) = \xi_1(y) \xi_2(\gamma)$. We have then
\begin{align*}
 (\pi(g)\xi) (y,\gamma) &= \xi(g^{-1}(y,\gamma)) = \xi(\gamma^{-1}\alpha(g^{-1},y)^{-1} (g^{-1}\cdot y))\\
&= \xi_1(g^{-1}\cdot y) (\lambda(\alpha(g^{-1},y))^{-1} \xi_2)(\gamma).
\end{align*}
Thus we have
\[
 \widetilde D(g)\xi = \xi_1(g^{-1}\cdot y)(D(\alpha(g^{-1},y)^{-1})\xi_2)(\gamma).
\]

\end{rk}

\begin{cor}
Let $p$ be such that the Burnside group $B(2,p)$ is not amenable. Then for any integer $n\ge 2$, $B(2,np)$ is not unitarizable. 
\end{cor}

\begin{proof}
Let $A$ be any infinite Abelian group that is $n$-periodic, i.e.\ such that $a^n=1$ for any $a$ in $A$. We can take for example $A = ({\bb Z}/n{\bb Z})^{({\bb N})}$. Let $G=B(2,p)$ and ${\cl G} =A \wr G$. We claim that $g^{np}=1$ for any $g$ in ${\cl G}$. To check this consider $g=a\gamma$ $(a\in A^{(G)},\gamma\in G)$. Then $g^2=a\gamma a\gamma = (aa_1)\gamma^2$ with $a_1=\gamma a\gamma^{-1}\in A^{(G)}$, $g^3 = (aa_1a_2)\gamma^3$ with $a_2=\gamma^2a\gamma^{-2}\in A^{(G)}$ and so on, so that
\[
 g^p = (aa_1\ldots a_p) \gamma^p=aa_1\ldots a_p\in A^{(G)}.
\]
Then since $A$ is $n$-periodic, $A^{(G)}$ is also $n$-periodic and hence $g^{np}=1$, proving our claim.\\
Now since it is countable and $(np)$-periodic, ${\cl G}$ is a quotient of $B(\infty,np)$ and, by Theorem \ref{Dixthm4-10},  it is not unitarizable. Therefore $B(\infty,np)$ itself  is not unitarizable (see Remark~\ref{rem}) and since $B(\infty,np)$ embeds as a subgroup in $B(2,np)$
(due to \v{S}irvanjan, 1976)   we conclude that $B(2,np)$ is not unitarizable.
\end{proof}

\begin{cor}
The following properties of a discrete group $G$ are equivalent:
\begin{itemize}
 \item[\rm (i)] $G$ is amenable.
\item[\rm (ii)] $A\wr G$ is unitarizable for some infinite Abelian group $A$.
\item[\rm (iii)] $A\wr G$ is unitarizable for all amenable groups $A$.
\end{itemize}
\end{cor}

\begin{proof}
The main point is (ii) $\Rightarrow$ (i) that follows from Theorem~\ref{Dixthm2-4} and Theorem~\ref{Dixthm4-10}. Then (iii) $\Rightarrow$ (ii) is trivial and (i) $\Rightarrow$ (iii) is easy:\ Indeed, if $G$ is amenable then $N\rtimes G$ is amenable for any amenable $N$, and $A$ amenable implies $A^{(G)}$ amenable. Thus $A\wr G$ is amenable (and hence unitarizable by Theorem~\ref{thm0.1}) if $A$ and $G$ are both amenable.
\end{proof}

\end{document}